\documentclass[twoside,maththm,reqno,tbtags]{amsart}
\usepackage[T1,T2A]{fontenc}
\usepackage[utf8]{inputenc}
\usepackage[ukrainian,russian,english]{babel}
\usepackage[colorlinks,urlcolor=magenta,linkcolor=blue,pagebackref]{hyperref}
\usepackage{balance}
\hypersetup{bookmarks=true,bookmarksopen=true,bookmarksopenlevel=4,pdfdisplaydoctitle,pdftitle=Inverse Variational Problem and Symmetry in Action: The Relativistic Third Order Dynamics,pdfauthor=Roman Ya. Matsyuk,pdfsubject=Higher order dynamics}
\usepackage{bm}
\makeatletter
\def\ps@firstpage{\ps@plain
  \def\@oddhead{\article@logo\hss}
}
\def\article@logo{\vbox to\headheight{\@parboxrestore \small
\rightline{\href{http://link.springer.com/book/10.2991/978-94-6239-109-3}{The Inverse Problem of the Calculus of Variations}}
\rightline{\href{http://link.springer.com/bookseries/13378}{Atlantis Studies in Variational Geometry}}
\rightline{Vol.~2 (2015) p.~75--102}
   \par\vss
}}
\makeatother
\newtheorem{proposition}{Proposition}
\theoremstyle{remark}
\newtheorem{remark}{Remark}
\let\cal\mathcal
\font\BBB=msbm10    
\font\ssb=cmssbx10 
\font\sssm=cmss8 
\font\bfrak=eufb10 
\newcommand{\eisf}[1]{\mbox{\sssm#1}}

\newdimen\colsep\colsep=\arraycolsep
\newcommand{\soft}[1]{#1{\kern-1.3pt}'{\kern-.3pt}}
\newcommand{\softL}{L{\kern-2pt}'}
\newcommand{\ms}{\kern.3em}
\newcommand{\nsp}{\mspace{-9mu}}
\newcommand{\RN}[1]{\uppercase\expandafter{\romannumeral#1}}
\newcommand{\NO}[1]{N$^{\rm o}$\kern.1em #1}
\newbox\stacka\setbox\stacka=\hbox{$\stackrel{\scriptstyle\downarrow}{{\mbox{\ssb A}}}$}
\edef\stack{\vrule width 0pt depth 0pt height \ht\stacka}
\def\sa{\copy\stacka\ht\stacka}
\def\sa{\copy\stacka}
\let\bf\bm
\newcommand{\bmat}[1]{{\bm{#1}}}
\newcommand{\A}{\mbox{\ssb A}}
\newcommand{\B}{\mbox{\ssb B}}
\newcommand{\CC}{\mbox{\ssb c}}
\newcommand{\V}{\mbox{\ssb v}}
\newcommand{\K}{\mbox{\ssb k}}
\newcommand{\Zero}{\mbox{\ssb 0}}
\newcommand{\bcdot}{\mbox{\boldmath$\cdot$}}
\newcommand{\bpartial}{{\bm\partial}}
\newcommand{\ubm}[1]{\mbox{\unboldmath$#1$}}
\newcommand{\bkey}[1]{{\boldsymbol#1}}
\newbox\abox\setbox\abox=\hbox{\ssb A}
\newbox\Elocity\setbox\Elocity=\hbox{\ssb v}
\newbox\PrimE\setbox\PrimE=\hbox{{\ssb v}$^{\prime}$}
\newbox\elocity\setbox\elocity=\hbox{\sssm v}
\newbox\Prime\setbox\Prime=\hbox{{\sssm v}$^{\scriptscriptstyle\prime}$}
\def\WW{\copy\PrimE}
\def\vv{\copy\elocity}
\def\ww{\copy\Prime}
\newcommand{\SSS}[1]{{\scriptscriptstyle#1}}
\newcommand{\au}{\|\bu\|}
\newcommand{\bfe}{\hbox{\ssb e}}
\newcommand{\eps}{\underline{\epsilon}\,}
\newcommand{\bvarepsilon}{{\boldsymbol\varepsilon}}
\newcommand{\bepsilon}{{\underline{\boldsymbol{\epsilon}}}}
\newcommand{\x}{{\sf x}}
\newcommand{\sv}{{\sf v}}

\newcommand{\sE}{{\sf E}}
\newcommand{\se}{{\sf e}}
\newcommand{\sA}{{\sf A}}
\newcommand{\sB}{{\sf B}}
\newcommand{\sk}{{\sf k}}
\newcommand{\bu}{{\bm u}}
\newcommand{\R}{\mbox{\BBB R}}
\newcommand{\BBN}{\mbox{\BBB N}}
\newcommand{\Dx}{{D}_{\x}}
\newcommand{\Du}{{D}_{\sv}}
\newcommand{\Dw}{{D}_{\sw}}
\newcommand{\bw}{\mbox{\ssb w}}
\newcommand{\sw}{{\sf w}}
\newcommand{\M}{\mbox{\ssb M}}
\newcommand{\sM}{{\sf M}}
\newcommand{\bb}{\mbox{\ssb b}}
\newcommand{\Sb}{{\sf b}}
\newcommand{\bd}{\,\bkey.\,}
\newcommand{\lp}[2]{\partial_{_{_{_{{{\scriptstyle #1}}}}}}{\!}_{#2}{}}
\newcommand{\bp}[1]{\bpartial_{#1}}
\newcommand{\dw}{{\sw'}}
\newcommand{\ddw}{{\sw''}}
\newcommand{\bE}{\mbox{\ssb E}}
\newcommand{\m}{\mbox{\ssb m}}
\newcommand{\N}{\mbox{\ssb N}}
\newcommand{\PP}{\mbox{\ssb P}}
\newcommand{\Q}{\mbox{\ssb Q}}
\newcommand{\pl}[6]{L
_{_{_{_{\scriptstyle #1}}}}{\!}_{#2}{}_{_{_{_{\scriptstyle #3}}}}{\!}_{#4}{}_{_{_{_{\scriptstyle #5}}}}{\!}_{#6}{}
 }
\newcommand{\X}{\mbox{\ssb x}}
\newcommand{\W}{\bm{\varOmega}}
\newcommand{\fa}{\mbox{\bfrak a}}
\newcommand{\bP}{\boldsymbol{\pi}}
\newcommand{\w}{\boldsymbol{\omega}}
\newcommand{\bdot}{\boldsymbol{\,.\,}}
\begin{document}
\author{{Roman~Ya. Matsyuk}}
\title[The Relativistic Third Order Dynamics]{Inverse Variational Problem and Symmetry in Action: The Relativistic Third Order Dynamics}
\address{\href{http://www.iapmm.lviv.ua/12/eng/files/st_files/matsyuk.htm}{Institute for Applied Problems in Mechanics \& Mathematics}
\\15 Dudayev St., 79005 Lviv, Ukraine \textit{and} \href{http://www.lepageri.eu/researchers/roman-matsyuk}{Lepage Research Institute}, Czech Republic}
\email{matsyuk@lms.lviv.ua, romko.b.m@gmail.com}
\begin{abstract}{Tools of the intrinsic analysis on manifolds, helpful in solving the invariant inverse problem of the calculus of variations are being presented comprising a combined approach which consists in the simultaneous imposition of symmetry principles and the inverse variational problem considerations in terms of vector-valued differential forms. In three-dimensional space-time we obtain a unique (covector) third-order Poincar\'e-invariant variational equation, which then is identified with the motion of a free relativistic top in flat three-dimensional space-time.}
\end{abstract}
\maketitle
\tableofcontents
\section{Introduction}
For the past four decades the subject of higher-order mechanics was
revisited by many authors from the point of view of global analysis,
including certain features of intrinsic differential geometry
(see monographs \cite{Krup}, \cite{Leon}, \cite{Saund} and the references therein).
The equations of motion and the Lagrange function of higher order analytical mechanics contain higher order derivatives. Such mechanics is often referred to as generalized mechanics or Ostrogradsky's mechanics. Ostrogradsky
was interested in variational principles with higher order derivatives present in the integrand~\cite{Ostrogradsky}.
As far as we know, he was the first to  introduce a hierarchy of generalized momenta corresponding to the sequence of higher order velocities in the Lagrange function (see~\cite{Leon} and the reference therein).
In the case when no transition from the Lagrangian framework to a generalized canonical framework involving generalized Ostrogradky's momenta is made, one may prefer to speak about the higher order dynamics.

The more intriguing fact is that the
investigations on the application of Ostrogradsky's mechanics to
real physical models have not been abandoned since the pioneer works by
Mathisson, Bopp, Weyssenhoff, Raabe, and H\"onl (see
references \cite{Bopp}, \cite{Hoenl}, \cite{Math}, \cite{Weyss}). Most of the applications consider models of test particles
endowed with inner degrees of freedom~\cite{Tulcz1}, \cite{Riewe}, \cite{Plyush}, \cite{Nest}, \cite{Leiko}, \cite{Arodz}, \cite{Ners}, \cite{Arreaga} or models that put the notion of the acceleration into the framework of general
differential-geometric structure of the extended configuration space of
the particle~\cite{Scarp}. One interesting example of how the derivatives
of the third order may appear in the equations of motion of a test particle is
provided by the Mathisson--Papapetrou equations
\begin{eqnarray}
  &&  \frac{D}{d\zeta}\left(m_{\SSS0}\frac{u^\alpha}{\au}+\frac{u_\gamma
}{\au^2}\frac{D}{d\zeta}S^{\alpha\gamma}\right) = {\cal F}^\alpha,
\label{MPu}
\\
&& \frac{D}{d\zeta}S^{\alpha\beta} = \frac{1}{\au^2}\left(u^\beta
u_\gamma\frac{D}{d\zeta}S^{\alpha\gamma}-u^\alpha
u_\gamma\frac{D}{d\zeta}S^{\beta \gamma}\right) \label{MPS}
\end{eqnarray}
together with the supplementary condition
\begin{equation}
u_\gamma S^{\alpha\gamma}=0.
\label{Pir}
\end{equation}
It is immediately clear that the second term in (\ref{MPu}) may produce
the derivatives of the third order of the space-time variables $x^\alpha$ as
soon as one replaces
$u_\gamma{\strut DS^{\alpha\gamma}}/{d\zeta}$ with $\displaystyle -
S^{\alpha\gamma}{\strut Du^\gamma}/{d\zeta}$ in virtue of (\ref{Pir}). Such a substitution in fact means differentiating equation (\ref{Pir}). However,
the system of equations thus obtained will not possess any additional
solutions comparing to that of \eqref{MPu}--\eqref{Pir} as long as one keeps the original constraint (\ref{Pir}). The system of equations \eqref{MPu}--\eqref{Pir} is discussed in~\cite{PlyaPhRev}.
The right-hand side of equation~(\ref{MPu}) vanishes in the absence of gravitation.

In this contribution we intend to present certain tools of the intrinsic analysis on manifolds that may appear helpful in solving the invariant inverse problem of the calculus of variations.
The main goal is to introduce a
combined approach consisting in the simultaneous imposition of symmetry principles and
the inverse variational problem considerations in terms
of vector-valued differential forms. Besides, a simple algorithm for an invertible transition from an autonomous variational problem to the variational problem in a parametric form is established.
In the special case
of three-dimensional space-time, we shall successfully follow some
prescriptions for obtaining third-order Poincar\'e-invariant variational
equations up to the very final solution, thus discovering the unique possible one.
The example
shows nonexistence of a globally and intrinsically defined
Lagrangian for Poincar\'e-invariant and well-defined unique
variational equation in the considered case. The model will then be identified with the motion of a free relativistic top in the flat three-dimensional space-time by means
of comparing it to equations \eqref{MPu}--\eqref{Pir} when
$R^\alpha{}_{\beta\delta\gamma}=0$ (and, consequently, ${\cal F}^\alpha=0$).

Our example exposes the following features of the inverse problem of the calculus ofvariations:
\begin{itemize}
\item Nonexistence (in our case) of a well-defined invariant
Lagrangian for intrinsically well-defined and Poincar\'e-invariant
\index{Poincar\'e! symmetry}\index{symmetry! Poincar\'e}
equations of motion.
\item
Existence of a finite set of (degenerate, in our case) Lagrangians, each producing the aforementioned equations.
\item
Failure of a sum of Lagrangians to be a Lagrangian of minimal order; in our case, this is caused by the dependence of different Lagrangians on different sets of second-order derivatives.
\item
Invariance, in our case, of this set of Lagrangians, but not the Lagrangians themselves, with respect to permutations of the axes of the Lorentz frame.
\end{itemize}
In order to simplify the exposition, some of the technical details are omitted in Sections~\ref{Section2}--\ref{Section5}. These details are elucidated  in Section~\ref{Appendix}.

\section{Homogeneous Form and Pa\-ra\-me\-t\-ric Invariance\label{Section2}}
Presentation of the equation of motion in the so-called `manifestly
covariant form' stipulates introducing the space of
Ehresmann's velocities over the configuration manifold $M$ of the particle,
$T^kM=\{x^\alpha,\dot x^\alpha,\ddot x^\alpha,\dots, x_{(k)}^\alpha\}$. In the following
the notations $u^\alpha$, $\dot u^\alpha$, $\ddot u^\alpha$,
$u_{(r)}^\alpha$ will frequently be used in place of $\dot x^\alpha,
\ddot x^\alpha, x_{(3)}^\alpha,\dotsc, x_{(r+\SSS1)}^\alpha$, and also
$x_{(\SSS0)}^\alpha$ sometimes will denote the quantity $x^\alpha$. We call a mapping $\zeta\mapsto x^\alpha(\zeta)$ the {\it parametrized\/}
(by means of
$\zeta$) {\it world line\/}. Its image in $M$ will be called
the {\it non-parameterized world
line.} The term `world line' is borrowed from the relativity theory, and the Reader may replace it with `path' or `curve'.  As we are interested in a variational equation (of order
$s$) that should describe the non-parameterized world lines of the
particle,
\begin{equation}
{\cal E}_\alpha\left(
x^\alpha,u^\alpha,\dot u^\alpha,\ddot
u^\alpha,\dots, u_{(s-\SSS1)}^\alpha
\right)=0,
\label{f}
\end{equation}
the Lagrange function ${\cal L}$  has to satisfy the Zermelo conditions,
which in our case of only derivatives up to the second order present in
${\cal L}$ read
\begin{align*}
u^\beta\frac{\partial}{\partial u^\beta}{\cal L}+2\dot
u^\beta\frac{\partial}{\partial \dot u^\beta}{\cal L} &= {\cal L},
\\
u^\beta\frac{\partial}{\partial \dot u^\beta}{\cal L} &= 0.
\end{align*}

In this approach the independent variable $\zeta$ (called the {\it parameter along
the world line}\/) is not included into the configuration manifold $M$. Thus
the space $T^kM$ is the appropriate candidate for the role of the
underlying manifold on which the variational problem in the autonomous
form should be posed. We may include the parameter $\zeta$ into the
configuration manifold by introducing the trivial fiber manifold
$\R\times M\to\R$, $\zeta\in\R$, and taking into consideration its
$k^{\rm th}$-order prolongation, $J^k(\R,M)$, i.e. the space of the $k^{\rm th}$-order jets of local cross-sections of
$Y=\R\times M$ over $\R$. Each such a cross-section of $Y$ is nothing but
the graph in $\R\times M$ of some local curve $x^\alpha(\zeta)$ in $M$.
For each $r\in\BBN$ there exists a natural projection
\begin{equation}
p_{\SSS0}^r:J^r(\R,M)\to T^rM,
\label{p0r}
\end{equation}
introduced as follows. The manifold $T^rM$ consists of the derivatives up to the
$r^{\rm th}$-order of curves $x^\alpha(\zeta)$ in $M$, evaluated at
$0\in\R$. For every $\tau\in\R$, denote by the same character $\tau$
the mapping $\zeta\mapsto\zeta+\tau$ of
$\R$ onto itself. Then the projection reads
\begin{eqnarray} 
\nonumber &&p_{\SSS0}^r:\biggl(\tau;x^\alpha(\tau),\frac{d}{d\zeta}x^\alpha(\tau),
\frac{d^2}{d\zeta^2}x^\alpha(\tau),\dots,\frac{d^r}{d\zeta^r}x^\alpha(\tau)\biggr)\nonumber
\\
&& {}\mapsto
\biggl(\,(x^\alpha\circ\tau)\,(0), \frac{d}{d\zeta}\,(x^\alpha\circ\tau)\,(0),\frac{d^2}{d\zeta^2}\,(x^\alpha\circ\tau)\,(0), \dots,
 \frac{d^r}{d\zeta^r}\,(x^\alpha\circ\tau)\,(0)\biggr).
\label{p0}
\end{eqnarray}

By means of projection (\ref{p0r}), every Lagrange function
${\cal L}$ initially defined on $T^kM$ may be pulled back to the manifold
$J^k(\R,M)$ and thus defines there the function ${\cal L}_{{\SSS0}}$ by the
formula ${\cal L}_{{\SSS0}}=\nobreak{\cal L}\circ p_{\SSS0}^k$. We say that the
differential form
\begin{equation}
\lambda={\cal L}_{\SSS0}d\zeta
\label{lambda0}
\end{equation}
constitutes a variational problem in the extended parametric form because in
the construction of the new configuration manifold $\R\times M$ the
independent variable $\zeta$ is artificially doubled. Nevertheless, we shall need
this construction later.

Let us return to the variational problem on the manifold $T^kM$ specified by a given Lagrange function ${\cal L}$. The moment we impose the
Zermelo conditions, the problem becomes degenerate. One can avoid this degeneracy by reducing the number of velocities at
the cost of losing the `homogeneity' property of the equation~(\ref{f}) as follows. Consider
some way of separating the variables $x^\alpha\in M$ into $t\in\R$
and $\x^i\in Q$, $\dim Q=\dim M-1$, thus making $M$ into some fibration,
$M\approx\R\times Q$, over $\R$. The manifold of jets $J^r(\R,Q)$
provides a local representation of what is known as the manifold
$C^r(M,1)$ of $r$-contact one-dimensional submanifolds of $M$.
Intrinsically defined global projection of non-zero elements of $T^rM$
onto the manifold $C^r(M,1)$, in this local and `non-covariant' representation is denoted by
\begin{equation}
\wp^r:T^rM\setminus\{0\}\to J^r(\R,Q),
\label{wp}
\end{equation}
and in the third order is implicitly defined by the following formulae,
where the local coordinates in $J^r(\R,Q)$ are denoted by $(t; \x^i,
\sv^i, {\sv'}^i, {\sv''}^i, \dots, \sv_{(r-\SSS1)}^i)$ with $\sv_{(\SSS0)}^i$ and $\sv_{(-\SSS1)}^i$
representing $\sv^i$ and $\x^i$:
\begin{align}
\nonumber
\dot t\, \sv^i  &= u^i
\\
(\dot t)^3 {\sv'}^i  &=  \dot t \dot u^i - \ddot t u^i
\label{calp}
\\
\nonumber
(\dot t)^5 {\sv''}^i  &=  (\dot t)^2 \ddot u^i - 3\, \dot t\, \ddot t\, \dot
u^i + \left[3(\ddot t)^2 -\dot t\; \dddot t\right] u^i.
\end{align}

There does not exist any well-defined projection from the manifold
$C^r(M,1)$ onto the space of independent variable $\R$, so an expression
\begin{equation}
\Lambda=L\big(t; \x^i, \sv^i, {\sv'}^i, {\sv''}^i, \dots, \sv_{(k-
\SSS1)}^i \big)dt
\label{L}
\end{equation}
will depend on the selection of the local representation $M\approx\R\times
Q$. We say that two different expressions of type~(\ref{L}) define the same variational problem in the parametric form if their difference
expands into nothing but the pull-backs to $C^k(M,1)$ of the
contact forms
\begin{equation}
\theta^i=d\x^i-\sv^idt
\label{Theta}
\end{equation}
that live on the manifold $C^{\SSS1}(M,1)$.
These differential forms obviously vanish along the jet of any curve
$\R\to Q$.

Let the components of the variational equation
\begin{equation}
\sE_i = 0
\label{fe}
\end{equation}
corresponding to the Lagrangian (\ref{L}) be treated as the components of the vector-valued one-form
\begin{equation}
\bfe=\big\{\sE_idt\big\}.
\label{e}
\end{equation}

We intend to give a `homogeneous' description to (\ref{e}) and
(\ref{L}) in terms of some objects defined on $T^sM$ and $T^kM$,
respectively. But we cannot apply directly the pull-back operation to the
Lagrangian~(\ref{L}) because the pull-back of one-form is a one-form
again, and what we need on $T^kM$ is a Lagrange {\em function} and not a
differential form. However, it is possible to pull (\ref{L}) ultimately all the way
back to the manifold $J^k(\R,M)$ using the composition of projections (\ref{p0r}) and (\ref{wp}),
\begin{equation}
p^k=\wp^k\circ p_{\SSS0}^k\,.
\label{p}
\end{equation}
This way we obtain a
differential form $\big(L\circ p^k\big)dt$. But what we do desire is a
form that solely involves $d\zeta$ (i.e., the semi-basic with respect
to the projection $J^k(\R,M)\to\R$). Fortunately, the two differential
forms, $dt$ and $\dot t d\zeta$, differ at most by the
contact form
\begin{equation}
\vartheta=dt-\dot t d\zeta.
\label{theta}
\end{equation}
Now, we recall that equivalent Lagrangians
that are of the same structure (\ref{L}) differ by multiplies of the contact
forms (\ref{Theta}). By
(\ref{p0}) and (\ref{calp}), the pull-backs of the
contact forms (\ref{Theta}) are superpositions of the contact forms (\ref{theta})
and
\begin{equation}
\vartheta^i=dx^i-u^i d\zeta,
\label{thetas}
\end{equation}
that is,
\[
p^{\SSS1*}\theta^i= dx^i - \big(\sv^i\circ p^{\SSS1}\big)dt = \vartheta^i -
\big(\sv^i\circ p^{\SSS1}\big)\vartheta.
\]
Thus, every variational problem posed on $J^k(\R,Q)$ and represented by (\ref{L}) transforms into an equivalent variational problem on $J^k(\R,M)$ characterized by
\begin{equation}
\lambda=\big(L\circ p^k\big)\dot t d\zeta.
\label{lambda}
\end{equation}
But the Lagrange function of this new variational
problem,
\begin{equation}
{\cal L}_{\SSS0}=\big(L\circ p^k\big)\dot t,
\label{L0}
\end{equation}
does not depend on the parameter $\zeta$, and, consequently, may be
thought of as a function defined on $T^kM$.

We prefer to cast the variational equation (of order $s\leq 2k$)
generated by the Lagrangian (\ref{lambda}) into the framework of
vector-valued exterior differential systems theory by introducing the
following vector-valued differential one-form
\begin{equation}
\bvarepsilon={\cal E}_\alpha \big(x^\alpha, \dot x^\alpha, \dots,
x_{(s)}^\alpha \big)d\zeta
\label{epsilon}
\end{equation}
defined on the manifold $J^s(\R,M)$.
The expressions ${\cal E}_\alpha \big(x^\alpha, \dot x^\alpha, \dots,
x_{(s)}^\alpha \big)$ in (\ref{epsilon}) may also be interpreted as the quantities defined on $T^sM$, similar to the interpretation of ${\cal L}_{\SSS0}$. Summarizing, we obtain following statement:
\begin{proposition}\label{homogen}
Let the differential form \eqref{e} correspond to the variational
equation with the Lagrangian \eqref{L}. Then the quantities
\begin{equation}
{\cal E}_\alpha = \left\{-u^i\sE_i, \dot t \,\sE_i \right\}
\label{CalE}
\end{equation}
are the variational expressions for the Lagrange function \eqref{L0}.
\end{proposition}

In this case the $s^{\rm th}$-order equation (\ref{f}) gives a
homogeneous representation of the same non-parameterized world lines of a particle
governed by the variational problem
(\ref{L0}), as does the equation (\ref{fe}) with the Lagrangian given by
(\ref{L}). Besides, ${\cal L}_{\SSS0}$ satisfies the Zermelo
conditions. For more details, we refer the Reader to paper~\cite{DGA2001}.

\section{The Criterion of Variationality}
\index{equation! Poincar\'e-invariant}\index{Poincar\'e! invariant equation}
Our main intention is to find a Poincar\'e-invariant ordinary (co-vector)
differential equation of the third order in three-dimensional space-time.
With this goal in mind we organize the expressions $\sE_i$ in (\ref{e})
into a single differential object, the exterior one-form
\begin{equation}
\se_{\SSS0}=\sE_i d\x^i
\label{e0}
\end{equation}
defined on the manifold $J^s(\R,Q)$, so that the vector-valued differential form
(\ref{e}) should now be viewed as a coordinate representation of the
intrinsic differential-geometric object
\begin{equation}
\se=\se_i d\x^i=\sE_i dt \otimes d\x^i = dt \otimes \se_{\SSS0}.
\label{efull}
\end{equation}
The differential form $\se$ constructed this way is an element of the graded module of semi-basic with respect to $\R$ differential forms
on $J^s(\R,Q)$ with values in the bundle of graded algebras $\wedge T^*Q$ of scalar forms on
$TQ$. Of course, due to the dimension of $\R$, only functions
(i.e. semi-basic zero-forms) and semi-basic one-forms (i.e. proportional to $dt$) exist. We also wish to mention
that every (scalar) differential form on $Q$ is naturally treated as a
differential form on $T^rQ$, i.e. as an element of the graded algebra of cross-sections of
$\wedge T^*\big(T^rQ\big)$.

For an arbitrary $s\in\BBN$, let $\Omega_s(Q)$ denote the algebra of (scalar) differential forms on
$T^sQ$ with coefficients depending on $\sv_{r-\SSS1}$,  $r\le s$, and $t\in \R$. It is
possible to develop calculus on $\Omega_s(Q)$ by introducing the
operator of vertical (with respect to $\R$) differential $d_{\sv}$
and the operator of total derivative (or formal time derivative) $D_t$ by setting\[
d_{\sv} f =\frac{\partial f}{\partial \x^i} d\x^i + \frac{\partial
f}{\partial \sv_{(r)}^i} d\sv_{(r)}^i, \quad d_{\sv}^2 = 0,
\]
so that $d_{\sv}\x^i$ and $d_{\sv}\sv_{(r)}^i$ coincide with $d\x^i$ and
$d\sv_{(r)}^i$, respectively, and
\[
D_tf=\frac{\partial f}{\partial t}+ \sv^i \frac{\partial f}{\partial
\x^i}+ \sv_{(r+\SSS1)}^i \frac{\partial f}{\partial \sv_{(r)}^i}, \quad
D_td_{\sv} = d_{\sv}D_t.
\]

There exists a notion of {\it derivation} in a graded algebra endowed with
a generalized commutation rule, as  $\Omega_s(Q)$ is. An operator $D$ is
called a derivation of degree $q$ if for any differential form $\varpi$
of degree $p$ and an arbitrary differential form $w$ the property
\[
D(\varpi \wedge w) = D(\varpi) \wedge w +(-1)^{pq} \varpi \wedge D(w)
\]
holds.
In order to complete the above definitions, it is necessary to require for $d_{\sv}$ to be
a derivation of degree $1$ and for $D_t$ to be a derivation of degree $0$.
Additionally, we need one more derivation of
degree $0$, denoted here as $\iota$ and defined by its action on
functions and one-forms that locally generate the algebra
$\Omega_s(Q)$:
\[
\iota f = 0,\quad \iota d\x^i = 0,\quad \,\iota d\sv^i = d\x^i,\quad
\iota d\sv_{(r)}^i = (r+1)\, d\sv_{(r-\SSS1)}^i.
\]

Let the operator $\deg$ denote the evaluation of degree of a differential form.
The {\em Lagrange differential} \index{Lagrange! differential} $\delta$ is initially introduced by its
action on the elements of $\Omega_s(Q)$,
\begin{equation}\label{matsyuk:delta}
\delta = \left(\deg + \sum_{m=1}^s\frac{(-1)^m}{m!}D_t{}^m \iota^m \right)d_{\sv},
\end{equation}
and next trivially extended to the entire graded module of semi-basic
differential forms on $J^s(\R,Q)$ with values in  $\wedge
T^*\big(T^rQ\big)$ by means of the formulae
\begin{align*}
\delta (\omega_i dt\otimes d\x^i) &= dt\otimes \delta (\omega_i d\x^i),
\\
\delta (\omega^r_i dt\otimes d\sv_{(r)}^i) &= dt\otimes \delta
(\omega^r_i d\sv_{(r)}^i).
\end{align*}

The operator $\delta$ possesses the property $\delta^2=0$ (see \cite{Lawruk} and~\cite{Kolar}).  For the differential geometric objects (\ref{efull}) and (\ref{L})
the following relation holds:
\begin{equation}
\se=\delta\Lambda=dt\otimes\delta L.
\label{edelta}
\end{equation}
Now the criterion for an arbitrary set of expressions $\big\{\sE_i\big\}$
in (\ref{e})  to be variational equations for some Lagrangian reads
\begin{equation}
\delta\se= dt\otimes\delta \se_{\SSS0} =0,
\label{crit}
\end{equation}
with $\se$ constructed from $\big\{\sE_i\big\}$ as in (\ref{e0})
and (\ref{efull}). The differential 2-form~$\delta \se_{\SSS0}$ (or its semi-basic extension~$\delta \se$) is sometimes called the \emph{Helmholtz form}\index{Helmholtz! form}\index{form! Helmholtz} related to the \emph{dynamical form}~$\se$~\cite{Krup}\index{form! dynamical}. An alternative way to say that a differential form~$\se$ is dynamical is to demand that $\,\iota\se=0$.

Of course, one may extend the above constructions to
analogous objects on the manifold $J^s(\R,M)$ in (\ref{p0r}) and
obtain the operator, the Lagrange differential $\delta^{\SSS Y}$ acting
on semi-basic, with respect to $\R$, differential forms on  $J^s(\R,M)$
with values in the bundle $\wedge T^*\big(T^sM\big)$. The operator $\delta^{\SSS Y}$ on the algebra
$\Omega_s(M)$  preserves the sub-algebra of forms
that do not depend on the parameter $\zeta\in\R$. The restriction of
$\delta^{\SSS Y}$ to the algebra of differential forms defined on
$T^sM$ will be denoted by $\delta^{\SSS T}$. It was introduced
in~\cite{Tulcz2}. If in (\ref{lambda0}) the Lagrange function ${\cal L}_{\SSS0}$
does not depend on the parameter $\zeta\in\R$, as is the case of
\eqref{lambda} and \eqref{L0}, rather than applying $\delta^{\SSS Y}$ to the forms $\lambda$
from (\ref{lambda0}) and
\begin{equation}
\varepsilon = \varepsilon_\alpha dx^\alpha ={\cal E}_\alpha d\zeta\otimes
dx^\alpha
\label{epsfull}
\end{equation}
from (\ref{epsilon}), we may apply the restricted operator  $\delta^{\SSS
T}$ to the Lagrange function ${\cal L}_{\SSS0}$ and to the differential
form
\begin{equation}
\varepsilon_{\SSS0} = {\cal E}_\alpha dx^\alpha.
\label{eps0}
\end{equation}
In the case of (\ref{L0}) the criteria $\delta^{\SSS Y}\varepsilon=0$,
\begin{equation}
\delta^{\SSS T}\varepsilon_{\SSS0}=0,
\label{criteps}
\end{equation}
and (\ref{crit}) are equivalent, and the variational equations,
produced by the expressions $\varepsilon=\delta^{\SSS Y}\lambda$ from
\eqref{epsfull} and \eqref{lambda}, $\varepsilon_{\SSS0}=\delta^{\SSS T}{\cal
L}_{\SSS0}$ from (\ref{eps0},~\ref{L0}), and $\se$ from (\ref{edelta}) are
equivalent to (\ref{f}). The expressions (\ref{e}) and (\ref{L}) are not
`generally covariant' whereas (\ref{eps0}) is. But the criterion
(\ref{criteps}) needs to be solved along with Zermelo conditions, whereas
(\ref{crit}) is self-contained.

The presentation of a system of variational expressions
$\big\{\sE_i\big\}$ under the guise of a semi-basic (i.e. in
$dt$ solely) differential form that takes values in the bundle of one-forms
over the configuration manifold $Q$ is quite natural:
\begin{itemize}
\item the Lagrange density (called {\it Lagrangian} in this work) is a
one-form in $dt$ only;
\item the destination of the Euler--Lagrange expressions in fact consists
in evaluating them on the infinitesimal variations, i.e. the vector
fields tangent to the configuration manifold $Q$  along the critical
curve; consequently, the set of $\sE_i$'s constitutes a linear form on the
cross-sections of  $TQ$ with the coefficients depending on higher
derivatives.
\end{itemize}

More details may be found in~\cite{Var}.

\section{The Lepagean Equivalent}
The system of partial differential
equations imposed on $\sE_i$ that arises from (\ref{crit}), takes more
tangible shape in the particular case of third-order Euler--Poisson (i.e., ordinary
Euler--Lagrange) expressions.
The most general form of the Euler--Poisson equation of the third order reads
\begin{equation}\label{4}
\A{\,\bkey.\,}\V^{{{\prime}}{{\prime}}}{\, +
\,}(\V^{{\prime}}{\!\bkey.\,}{\bpartial}_{\eisf v})\,
\A{\,\bkey.\,}\V^{{\prime}}{\, + \,}\B{\,\bkey.\,}\V^{{\prime}}{\, +
\,}\CC\, = \,\Zero\,,
\end{equation}
where the notations of Proposition~\ref{prop3order} of Section~\ref{Appendix} are used and the conditions~(\ref{5}) on the skew-symmetric matrix $\A$, matrix~$\B$, and row $\CC$ hold. Hereafter the lower dot denotes contractions (such as multiplication of a row-vector by a column-vector).

Due to the affine structure of the left-hand side of equation (\ref{4}),
we may
introduce the differential form with coefficients that do not depend on the third-order derivatives:
\begin{align}
\nonumber
&\eps =
\sA_{ij} d\sv^{\prime}{}^j\otimes d\x^i + \sk_i dt \otimes
d\x^i,
\\
&\,\K =
(\WW{\!\bkey.\,}\bpartial_{\vv})\,\A{\,\bkey.\,}\WW+\B{\,\bkey.\,}\WW+\CC\,.
\label{6}
\end{align}

From the point of view of searching only holonomic local curves on
$J^3(\R,Q)$, the exterior differential systems that differ at most by  multiples of contact forms (\ref{Theta}) and multiples of
\[
\theta'{}^i=d\sv^i-\sv'{}^idt, \quad \theta''{}^i=d\sv'{}^i-\sv''{}^idt
\]
are considered equivalent. The differential forms (\ref{6}) and
(\ref{efull}) are equivalent:
\[
\eps-\se= \sA_{ij}\theta''{}^j\otimes d\x^i.
\]
The differential form (\ref{6}) may be accepted as an alternative
representation of the {\it Lepagean equivalent}~\cite{Krup} of
(\ref{efull}). \index{Lepagean equivalent}

\section{The Invariant Euler--Poisson Equation\label{Section5}}
Our primary interest is in the variational equations with symmetry.
Let $X(\bepsilon)$ denote the component-wise action of an infinitesimal
generator $X$ on a vector-valued differential form $\bepsilon$. Let the exterior differential system, generated by the form $\bepsilon$, possesses symmetry with the generator $X$. This means that there exist matrices ${\bf\varPhi}$,
${\bf\varXi}$, and ${\bf\varPi}$ that depend on $\copy\Elocity$ and
$\copy\PrimE$, and such that
\begin{equation}\label{7}
\ubm X(\bepsilon)={\bf\varPhi}\,.\,\bepsilon + {\bf\varXi}\,.\,(\ubm d\hbox{\ssb x}-
 \V\,\ubm d\ubm t)
+ {\bf\varPi}\,.\,(\ubm d  \V - \WW\ubm d \ubm t)\ubm.
\end{equation}

Equation (\ref{7}) specifies the condition that the vector exterior
differential system generated by the vector-valued differential form
$\bepsilon$ and the vector exterior
differential system generated by the shifted form $X(\bepsilon)$
are algebraically equivalent. For systems generated by one-forms, as in
our case, this is equivalent to the property that the set of
local solutions is preserved by the one-parametric Lie subgroup
generated by $X$, because any Pfaff system is complete (see \cite[page~64]{Slebodz} or~\cite[page~232]{Choquet-Analysis}). We see two advantages of this method:
\begin{itemize}
\item the symmetry concept is formulated in reasonably general terms;
\item the problem of invariance of a differential equation is reformulated
in algebraic terms by means of undetermined coefficients ${\bf\varPhi}$,
${\bf\varXi}$, and ${\bf\varPi}$;
\item the order of the underlying non-linear manifold is reduced
($J^2(\R,Q)$ instead of $J^3(\R,Q)$).
\end{itemize}

Further details may be found in~\cite{Sym}.

In the case of the Poincar\'e group\index{Poincar\'e! group}  \index{group! Poincar\'e} we assert that $\A$ and $\K$ in
(\ref{6}) do not depend upon $t$ and $\hbox{\ssb x}$. And for the sake of
reference it is worthwhile to put down the general expression of the generator
of the Lorentz group parameterized by a skew-symmetric matrix
${\W}$
and a
vector
$\bP$:
\begin{equation}\label{matsyuk:Lorentz}
\begin{split}
X = &
{}-(\bP\cdot\hbox{\ssb
x})\,{\partial_t}+{g_{\SSS0\SSS0}\,t}\,\bP\,.\,\bpartial_{\eisf x}
+\W\cdot(\hbox{\ssb x}\wedge\bpartial_{\eisf x})
\\ &{}+{g_{\SSS0\SSS0}}\,\bP\,.\,\bpartial_{\vv}
+(\bP\cdot \V)\:\V\,.\,\bpartial_{\vv}+\W\cdot(\V\wedge\bpartial_{\vv})
\\ &{}+2\,(\bP\cdot \V)\:\WW.\,\bpartial{_{\ww}}+(\bP\cdot
\WW)\:\V\,.\,\bpartial{_{\ww}}+\W\cdot(\WW\wedge\bpartial{_{\ww}})\,.
\end{split}
\end{equation}
Here, the centerdot denotes the inner product of vectors or
tensors and the lower dot denotes the contraction of row-vectors
and column-vectors, as before.

The system of equations \eqref{7} and \eqref{5} may possess many solutions, or no solutions
at all, depending on the dimension of the configuration manifold. For
example, in dimension \emph{one} any skew-symmetric matrix is zero.
If the dimension of manifold~$Q$ equals \emph{three}, there are no solutions to the system of PDEs
\eqref{7} and \eqref{5} (see Proposition~\ref{NoExist} in Section~\ref{Appendix}).
However, if the dimension of the manifold~$Q$ is \emph{two}, the
solution exists {\em and is unique}, up to a single scalar parameter
$\mu$ (see~\cite{Condenced}.

\begin{proposition}\label{matsyuk:E3}
The invariant Euler--Poisson equation\index{Euler-Poisson! equation} \index{equation! Euler-Poisson} for a relativistic two-dimen\-sional
motion is
\begin{align}
-\frac{\ast\V''\strut}{(1+\V\bcdot\V)^{3/2}\strut}
&+
3\,\frac{\ast\WW\strut}{(1 + \V\bcdot\V)^{5/2}}\,(\V\bcdot\WW)
\nonumber
\\
&-
\frac{\mu\strut}{(1+\V\bcdot\V)^{3/2}}\,\bigl[(1+\V\bcdot\V)\,\WW-
(\WW\bcdot\V)\,\V\bigr]=\Zero.
\label{10}
\end{align}
\end{proposition}
The dual vector $\ast\mbox{\ssb
v}$ above is defined, as usual, by the formula $(\ast\mbox{\ssb
v})_{i}=\epsilon_{ji}{\sf v}^{j}$.
For the proof of Proposition~\ref{matsyuk:E3}, see Section~\ref{Appendix}.

There are two different Lagrange
functions known for the left-hand side of (\ref{10}):
\begin{align}
L_{(1)} &= -\frac{\sv\,'\,^2\,\sv^1}{\sqrt{1+\sv_i\sv^{i}\strut }(1+\sv_2\sv^2)} + \mu
\sqrt{1+\sv_i\sv^i},
\label{L1}
\\
L_{(2)} &= \frac{\sv\,'\,^1\,\sv^2}{\sqrt{1 + \sv_i\sv^{i}\strut }(1 +
\sv_1\sv^1)} + \mu \sqrt{1 + \sv_i\sv^i}  \label{L2}.
\end{align}

With the help of Proposition~\ref{homogen}, we immediately
obtain the
`homogeneous' counterpart of~(\ref{10}),
\begin{equation}
-\frac{{\ddot{\bmat u}}\times\bmat u\strut}{\|\bmat u\|^{{\scriptstyle3}}\strut}
+ 3\,\frac{{\dot{\bmat u}} \times\bmat u\strut}{\|\bmat u\|^{{\scriptstyle5}}}\,({\dot{\bmat u}}\cdot\bmat u) - {} \frac{\mu\strut}{\|\bmat u\|^{{\scriptstyle3}}}\,\bigl[ (\bmat u\cdot\bmat u)\,{\dot{\bmat u}}
- ({\dot{\bmat u}} \cdot\bmat u)\,\bmat u  \bigr] =\bmat 0,
\label{homvareq}
\end{equation}
with the corresponding family of Lagrange functions,
\begin{eqnarray*}
{\cal L}_{(0)} & = & \frac{u^0\big(\dot u^2 u^1-\dot u^1 u^2\big)}{\au
\big(u_1 u^1 + u_2 u^2 \big)} + \mu\au\,,
\\
{\cal L}_{(1)} & = & \frac{u^1\big(\dot u^0 u^2-\dot u^2 u^0\big)}{\au
\big(u^0 u^0 +u_2 u^2 \big)} + \mu\au\,,
\\
{\cal L}_{(2)} & = & \frac{u^2\big(\dot u^1 u^0-\dot u^0 u^1\big)}{\au
\big(u_0 u^0 + u_1 u^1 \big)} + \mu\au\,,
\end{eqnarray*}

To produce a variational equation of the third order, the Lagrange function
should be of affine type in second derivatives. It makes no sense to even try
to find such a Poincar\'e-invariant Lagrange function in the space-time of  dimension
greater than two~\cite{thesis}. But the generalized momentum
\[
\frac{{\partial\cal L}}{\partial\bmat u} - \frac{ d}{{
d\zeta}}\frac{{\partial\cal L}}{\partial
\dot{\bmat u}
}
= \frac{{\dot{\bmat u}}\times\bmat u}{\|\bmat u\|^{{\scriptstyle3}}} + \ubm\mu\,\frac{\bmat u}{\|\bmat u\|}
\]
does not depend on a particular choice of the Lagrange
functions from the above family. This expression for the generalized momentum was (in different
notations) in fact obtained in~\cite{Plyush} by means of introducing an
abundance of Lagrange multipliers into the formulation of the corresponding
variational problem.

\section{An Example: Free Relativistic Top in Two Dimensions}
Equation~(\ref{homvareq}) carries certain amount of physical sense.
After some development it is possible to show
(see~\cite{DAN282} and~\cite{Astro}) that in terms of spin
vector,
\begin{equation*}
\sigma_{\alpha}=\frac{1}{2\au}\varepsilon_{\alpha\beta\gamma\delta}u^\beta
S^{\gamma\delta},
\end{equation*}
the Mathisson-Papapetrou equations\index{Mathisson-Papapetrou equations} \index{equations! Mathisson--Papapetrou} \eqref{MPu} and \eqref{MPS} under
the Mathisson--Pirani auxiliary condition (\ref{Pir}) \index{Mathisson--Pirani condition} \index{condition! Mathisson--Pirani} are equivalent to the following system:
\begin{align}
\varepsilon_{\alpha\beta\gamma\delta}\ddot u^\beta u^\gamma
\sigma^\delta -
3\,\frac{ { {\dot{\bmat u}} \!\cdot\!\bmat u} }{\phantom{^2}\au^2}
\,\varepsilon_{\alpha\beta\gamma\delta}\dot u^\beta u^\gamma \sigma^\delta
+ \frac{m_{\SSS0}}{\sqrt{|g|}} \left[{({\dot{\bmat u}}
\cdot\bmat u)}\,u_\alpha - \au^2 \dot u_\alpha\right]
&= {\cal F}_\alpha,
\label{18}
\\
\au^2 \dot\sigma_\alpha +{(\bmat\sigma\cdot {\dot{\bmat u}})}\, u_\alpha  &= 0,
\nonumber
\\
\bmat\sigma\cdot\bmat u &= \bmat0\,.\nonumber
\end{align}
Evidently, the four-vector $\bmat\sigma$ is constant if the force $\cal F_\alpha$ vanishes. Equation (\ref{18}) admits a
planar motion for which $u_3=\dot u_3=\ddot u_3=0$. In this case it becomes
\begin{equation*}
\ubm{\eta_3\sigma_3}\left(\frac{
\ddot{\bmat u}
\times\bmat u\strut}{\|\bmat u\|^{{\ubm{\scriptstyle3}}}} -
3\,\frac{
\dot{\bmat u}
\times\bmat u\strut}{\|\bmat u\|^{\ubm{\scriptstyle5}}}\,(
\dot{\bmat u}
\cdot\bmat u)\right) + {}
 \frac{\ubm{m_{\SSS0}}\strut}{\|\bmat u\|^{\ubm{\scriptstyle3}}}\,\bigl[(\bmat u\cdot \bmat u)\,
\dot{\bmat u}
 -
(
\dot{\bmat u}
\cdot\bmat u)\,\bmat u  \bigr] = 0,
\end{equation*}
where we set $g_{\alpha\beta}={\rm diag}\big( 1, \eta_1, \eta_2, \eta_3
\big)$. Comparison with (\ref{homvareq}) implies
\[
\mu=\frac{\strut m_{\SSS0}}{\eta_3\sigma_3}.
\]

\section{The Inverse Problem for the Euler--Poisson Equations\label{Appendix}}
\subsection{The Generalized Helmholtz Conditions}
We begin by presenting the coordinate form of the criterion
\begin{equation}\label{matsyuk:deltaEpsZero=0}
\delta \se_{\SSS0} =0.
\end{equation}
Recall that this criterion follows from~\eqref{crit}.
\begin{proposition}
Let
\begin{equation}\label{matsyuk:E0}
\big\{\sE_i=0\big\}
\end{equation}
be an arbitrary system of ordinary differential equations of order~$s$. The necessary and sufficient condition for system~(\ref{matsyuk:E0}) to be a system of Euler--Poisson expressions \index{Euler-Poisson! expressions}\index{expressions! Euler-Poisson} for some Lagrange function are given by the following system of partial differential equations imposed on the functions~$\sE_i$\textup{:}
\begin{subequations}\label{matsyuk:2}
\begin{align}
\label{matsyuk:2.1}
 \dfrac{\partial\sE_i}{\partial \x^j}
 - \dfrac{\partial\sE_j}{\partial \x^i}
+\sum_{k=0}^{s}(-1)^k D_t^k
\left(
\dfrac{\partial\sE_i}{\partial \sv_{k-1}^j}
-\dfrac{\partial\sE_j}{\partial \sv_{k-1}^i}
\right) &=0,&
\\
\label{matsyuk:2.2}
\dfrac{\partial\sE_i}{\partial
\sv_{{r}-1}^j}-\sum_{k={r}}^{s}(-1)^k\dfrac{k!}{(k-{r})!{r}!}D_t^{k-{r}}\dfrac{\partial\sE_j}{\partial
\sv_{k-1}^i}&=0,\quad 1\leq {r}\leq {s}.
\end{align}
\end{subequations}
\end{proposition}
\begin{proof}
We calculate the left hand side of~(\ref{matsyuk:deltaEpsZero=0}) in a way similar to that of~\cite[page~220]{Lawruk}, but provide more details. With~$\se_{\SSS0}$ from~(\ref{e0}) and~$\delta$ from~(\ref{matsyuk:delta}), one obtains
\begin{align*}
d_{\sv}\se_{\SSS0} &=
\sum_{p=0}^s\frac{\partial\sE_i}{\partial\sv_{p-1}^j}\,d\sv_{p-1}^j\wedge d\x^i,
\\[2pt]
\deg d_{\sv}\se_{\SSS0} &=
2\,d_{\sv}\se_{\SSS0},
\end{align*}
and thus, for $s\ge m\ge 1$,
\begin{equation*}
    \iota^m d_{\sv}\se_{\SSS0}=\sum_{p=m}^s\frac{p!}{(p-m)!}\,\frac{\partial\sE_i}{\partial
\sv_{p-1}^j}\,d\sv_{p-m-1}^j\wedge d\x^i.
\end{equation*}
Since for arbitrary differential forms $\alpha$, $\beta$, $\gamma$,
\begin{equation*}
    D_t^m(\alpha\wedge\beta\wedge\gamma)=\sum_{v +l+n=m}\frac{m!}{v!l!n!}D_t^v \alpha\wedge D_t^l\beta\wedge D_t^n\gamma\,,
\end{equation*}
one has
\begin{equation*}
D_t^m\left(\frac{\partial\sE_i}{\partial\sv_{p-1}^j}\,d\sv_{p-m-1}^j\wedge d\x^i\right) =
\sum_{v+l+n=m}\frac{m!}{v!l!n!}\left(D_t^v\frac{\partial\sE_i}{\partial\sv_{p-1}^j}\right)d\sv_{p-m-1+l}^j\wedge d\sv_{n-1}^i,
\end{equation*}
and, extending the range of $m$ to include $m=0$,
\begin{multline*}
\delta \se_{\SSS0}=d_{\sv}\se_{\SSS0}
\\
+
\sum_{m=0}^s\sum_{p=m}^s(-1)^m\frac{p!}{(p-m)!}\, \sum_{v +l+n=m}\frac{1}{v!l!n!}D_t^v \frac{\partial\sE_i}{\partial\sv_{p-1}^j}\,d\sv_{p-m+l-1}^j\wedge d\sv_{n-1}^i.
\end{multline*}
Let us introduce new summation indices $q$ and $u$ by $p-m=q$, $q+l=u$. Rewriting the above sum, we get
\setlength{\multlinegap}{0pt}
\begin{multline*}
\delta \se_{\SSS0}=
d_{\sv}\se_{\SSS0}
\\
+ \sum_{m=0}^s\sum_{q=0}^{s-m}(-1)^m\frac{(m+q)!}{q!}\sum_{v +u+n=m+q}\frac{1}{v!(u-q)!n!} D_t^v \frac{\partial\sE_i}{\partial\sv_{q+m-1}^j}\,d\sv_{u-1}^j\wedge d\sv_{n-1}^i,
\end{multline*}
where the indices $u$ and $q$ satisfy the condition $u-q\ge 0$. Next we transform the sums as follows,
\begin{align*}
    \sum_{m=0}^s\sum_{q=0}^{s-m}&=\sum_{k=0}^s\sum_{q+m=k},
\intertext{with $k=m+q$ without restrictions, and introducing $r=u+n$,  and noticing that~$r$ can not exceed $k$ in the first sum below and that the index $v$ in there has to satisfy the condition $v+r=k$, so that}
    \sum_{v +u+n=m+q}&=\sum_{r=0}^k\sum_{u+n=r}.
\end{align*}
Next, we replace $v$ with $k-r$ in the expression for $\delta \se_{\SSS0}$:
{\setlength{\multlinegap}{0pt}
\begin{multline*}
\delta \se_{\SSS0}=d_{\sv}\se_{\SSS0}
\\
+ \sum_{k=0}^s\sum_{q+m=k}(-1)^m\frac{k!}{q!}\sum_{r=0}^k\sum_{u+n=r}\frac{1}{(k-r)!(u-q)!n!} D_t^{(k-r)}\frac{\partial\sE_i}{\partial\sv_{k-1}^j}\,d\sv_{u-1}^j\wedge d\sv_{n-1}^i.
\end{multline*}}
Actually, the last sum is being carried out over~$u$ alone under the condition that each occurrence of $n$ is replaced with $r-u$:{\setlength{\multlinegap}{0pt}
\begin{multline*}
\delta \se_{\SSS0}=d_{\sv}\se_{\SSS0}
\\
+ \sum_{k=0}^s\sum_{q+m=k}\sum_{r=0}^k\sum_{u=0}^r\frac{(-1)^m u!k!}{u!(k-r)!(r-u)!q!(u-q)!}D_t^{(k-r)}\frac{\partial\sE_i}{\partial\sv_{k-1}^j}\,d\sv_{u-1}^j\wedge d\sv_{r-u-1}^i,
\end{multline*}}%
where we introduced the trivial multiplier $1\equiv{u!}/{u!}$ in the summand.
The sum over $q$ and $m$ may be carried out first and the index~$q$ runs from $0$ to $u$, since in accordance with other conditions, $u\le r$ and $r\le k$:
\begin{equation*}
    \sum_{q+m=k,\; u-q\ge0}(-1)^m\frac{u!}{q!(u-q)!}=\sum_{q=0}^u(-1)^{(k-q)}\binom{u}{q}.
\end{equation*}
This sum equals~$0$ if $u>0$, and~$(-1)^k$ if $u=0$. Finally,
\begin{equation*}
    \delta \se_{\SSS0}=d_{\sv}\se_{\SSS0} +\sum_{r=0}^s\sum_{k=0}^s\frac{(-1)^k k!}{(k-r)!r!}D_t^{(k-r)}\frac{\partial\sE_i}{\partial\sv_{k-1}^j}\,d\x^j\wedge d\sv_{r-1}^i,
\end{equation*}
where again the indices $k$ and $r$ are subject to the constraint $k-r\ge0$, so that the right-hand side may be rewritten as
\begin{equation*}
    \delta \se_{\SSS0}=\sum_{r=0}^s\left(\frac{\partial\sE_i}{\partial\sv_{r-1}^j}
    -
    \sum_{k=r}^s(-1)^k\frac{k!}{(k-r)!r!}D_t^{(k-r)}\frac{\partial\sE_j}{\partial\sv_{k-1}^i}\right)d\sv_{r-1}^j\wedge d\x^i.
\end{equation*}
The desired result is obtained by equating this expression to zero.
\end{proof}

\begin{proposition}
System of equations~\eqref{matsyuk:2} is equivalent to the system
\begin{equation}\label{matsyuk:criterion}
\dfrac{\partial\sE_i}{\partial
\sv_{{r}-1}^j}-\sum_{k={r}}^{s}(-1)^k\dfrac{k!}{(k-{r})!{r}!}D_t^{k-{r}}\dfrac{\partial\sE_j}{\partial
\sv_{k-1}^i} =0, \quad  0\le r\le s,
\end{equation}
obtained from~\eqref{matsyuk:2.2}
by extending the range of ${r}$ to
include ${r}=0$:
\end{proposition}
\begin{proof}
The antisymmetrization of~(\ref{matsyuk:criterion}) at ${r}=0$ produces  equation~(\ref{matsyuk:2.1}).
Conversely, in equation~(\ref{matsyuk:2.1}) let us separate the term with
$k=0$:
 \begin{equation}\label{matsyuk:2.1.3}
 2\,\dfrac{\partial\sE_i}{\partial \x^j}
 - 2\,\dfrac{\partial\sE_j}{\partial \x^i}
+\sum_{k=1}^{s}(-1)^k D_t^k
\dfrac{\partial\sE_i}{\partial \sv_{k-1}^j}
-\sum_{k=1}^{s}(-1)^k D_t^k\dfrac{\partial\sE_j}{\partial \sv_{k-1}^i}
 =0.
\end{equation}
In the first sum substitute the value of
${\partial\sE_i}/{\partial \sv_{k-1}^j}$ from
equation~(\ref{matsyuk:2.2}):
\[
\sum_{k=1}^{s}(-1)^k D_t^k \dfrac{\partial\sE_i}{\partial \sv_{k-1}^j}
=\sum_{k=1}^{s}(-1)^k D_t^k
\sum_{{r}=k}^{s}(-1)^{r}\dfrac{{r}!}{({r}-k)!k!}D_t^{{r}-k}\dfrac{\partial\sE_j}{\partial
\sv_{{r}-1}^i}.
\]
Now interchange the order of summation:
$\sum_{k=1}^{s}\sum_{{r}=k}^{s}=\sum_{\substack{{r}\ge k} =1}^{s}=\sum_{{r}=1}^{s}\sum_{k=1}^{r}$. Calculate the sum
over $k$:
\[
\sum_{k=1}^{r}(-1)^k\dfrac{{r}!}{({r}-k)!k!}=\sum_{k=0}^{r}(-1)^k \binom {r} k -
\binom {r} 0 =0-1=-1.
\]
Ultimately, equation~(\ref{matsyuk:2.1})
becomes
\[ 2\,\dfrac{\partial\sE_i}{\partial \x^j} -
2\,\dfrac{\partial\sE_j}{\partial \x^i} -\sum_{k=1}^{s}(-1)^k D_t^k
 \dfrac{\partial\sE_j}{\partial \sv_{k-1}^i} -\sum_{k=1}^{s}(-1)^k
D_t^k\dfrac{\partial\sE_j}{\partial \sv_{k-1}^i} =0,
\]
which
coincides with equation~(\ref{matsyuk:criterion}) multiplied by $2$ at $r=0$.
\end{proof}
  The
criterion~(\ref{matsyuk:criterion}) has been obtained by many authors. The Reader may
refer to book~\cite{Krup} by Olga Krupkov\'a for a review.
\subsection{The Fourth Order Variational ODEs}
When convenient, we shall use notations and conventions from matrix algebra, thus
the $\otimes$ symbol will mean the tensor (sometimes called
`direct') product of matrices, whereas $\odot$ and~$\wedge$ will respectively denote the symmetrization and the antisymmetrization of~$\otimes$.
We also introduce a special notation $\bw=\V'$ just for convenience.
Further, let us introduce the following special notations for the truncated total derivative operators:
\begin{align*}
{\Dx}&=\partial_{t}{\,+\,}\V{\,\bkey.\,}{\bpartial}_{\eisf x},
\\
\Du&=\partial_{t}{\,+\,}\V{\,\bkey.\,}{\bpartial}_{\eisf x}+\bw{\,\bkey.\,}{\bpartial}_{\eisf v},
\\
\Dw &=\partial_{t}{\,+\,}\V{\,\bkey.\,}{\bpartial}_{\eisf x}
 +\bw{\,\bkey.\,}{\bpartial}_{\eisf v}
 +\bw'{\,\bkey.\,}{\bpartial}_{\eisf w}.
\end{align*}
\begin{proposition}[\cite{matmet20}]\label{prop4order}
A system of the fourth order ODEs
\begin{equation*}
    \sE_i\left(t,\x^{j},\sv^{j},\sw^{j},\sw'^{j},\sw''^{j}\right)=0
\end{equation*}
is an Euler--Poisson system (in other words \emph{is a variational system}) if and only if it is of the form
\begin{equation*}
    {\bE}=\M\bd\bw''+(\bw'\!\bd\bpartial_{\sw})\M\bd\bw'+\A\bd\bw'+2\Du\M\bd\bw'+\bb,
\end{equation*}
where the symmetric matrix $\M$, the skew-symmetric matrix $\A$, and the row $\bb$ depend on the variables $t$, $\x^{j}$, $\sv^{j}$, $\sw^{j}$ and satisfy the following system of PDEs:
\begin{subequations}
\begin{gather}
    \lp\sw{[i}\sM_{j]k}=0\,,\notag\\
    \lp\sw i\sA_{jk}+2\lp\sv{[j}\sM_{k]i}=0\,,\tag{\theequation\,j}\label{j}\\
    \lp\sv{[i}\sA_{jk]}=0\,,\tag{\theequation\,i}\label{i}\\
    2\lp\sw{[i}\Sb_{j]}+3\Du\sA_{ij}=0\,,\tag{\theequation\,ii}\label{ii}\\[1\jot]
\begin{aligned}
         \lp\sw i\lp\sw j\Sb_k
         +2\lp\sv{(i}\sA_{j)k}
         -6\lp\x{(i}\sM_{jk)}+\Du\lp\sv k\sM_{ij}
&\\
         -4\lp\sv{(i}\sM_{j)k}-2\Du^2\lp\sw i\sM_{jk}&=0\,,
\end{aligned}
        \tag{\theequation\,jjj}\label{jjj}
 \\[1\jot]
 \begin{aligned}
    2\lp\sw k\lp\sv{[i}\Sb_{j]}-4\lp\x{[i}\sA_{j]k}
    +\lp\x j\sA_{ij}
    &
    +2\Du\lp\sv j\sA_{ij}
    \\
    &
    -4\Du\lp\x{[i}\sM_{j]k}-2\Du^2\lp\sv{[i}\sM_{j]k}=0\,,
    \end{aligned}\tag{\theequation\,iv}\label{iv}
    \\[1\jot]
    \bpartial_\sv\odot\bb-\Du\,\bpartial_\sw\odot\bb+\Du^3\M=\,\Zero\,,\tag{\theequation\,v}\label{v}\\
    4\bpartial_\x\wedge\bb-2\Du\,\bpartial_\sv\wedge\bb-\Du^3\A=\,\Zero\,.\tag{\theequation\,vi}\label{vi}
\end{gather}
\end{subequations}
\end{proposition}
\begin{proof}
Equations~(\ref{matsyuk:2}) become
\begin{align}
    \bp\ddw\wedge\bE &= \Zero,\label{diss8.2}
    \\
    \bp\dw\odot\bE-2D_t\bp\ddw\otimes\bE &= \Zero,\label{diss8.3}
    \\
    2\bp\sw\wedge\bE-3D_t\bp\ddw\otimes\bE+6D_t^2\bp\ddw\otimes	\bE &= \Zero,\label{diss8.4}
    \\
    2\bp\sv\odot\bE-2\bp\sw\otimes\bE+3D_t^2\bp\ddw\otimes
    \bE-4D_t^3\bp\ddw\otimes\bE &= \Zero,\label{diss8.5}
    \\
    2\bp\x\wedge\bE-D_t\bp\sv\wedge\bE+D_t^2\bp\sw\wedge\bE-
    D_t^3\bp\dw\wedge\bE+D_t^4\bp\ddw\wedge\bE &= \Zero.\label{diss8.6}
\end{align}
The second order derivatives of $\bE$ with respect to the variable $\bw''$ appear in~(\ref{diss8.3}) as the coefficients of $\bw'''$, so they should be zero. Together with~(\ref{diss8.2}) this implies the linear dependence,
\begin{equation*}
    \bE=\M\bd\bw''+\m
\end{equation*}
with a symmetric matrix $\M$, which allows us to further decouple \eqref{diss8.3}--\eqref{diss8.6} into separate equations according to the powers of the variable $\bw''$. Next, the terms containing $\bw'''$ in~(\ref{diss8.4}) and $\bw''''$ in~(\ref{diss8.5}) are, respectively,
\begin{align*}
3\bp\dw\otimes\bigl(\M\bd\bw'''\bigr) - 6\bigl(\bw'''\bd\bp\dw\bigr)\M
&= \Zero,
\\
3\bp\dw\otimes\bigl(\M\bd\bw''''\bigr) - 4\bigl(\bw''''\bd\bp\dw\bigr)\M
&= \Zero,
\end{align*}
implying that the matrix $\M$ does not depend on the variable $\bw'$, which simplifies the form of the equations (\ref{diss8.3}), (\ref{diss8.4}), and (\ref{diss8.5}):
\begin{equation}\label{diss8.3.3}
    \bp\dw\otimes\m-2\Dw\M=\,\Zero\,,
\end{equation}
\begin{multline}\label{diss8.4.4}
    2\bp\sw\wedge\m-3\Dw\bp\dw\otimes\m+6\bigl(\bw\bd\bp\x\bigr)\M
    +6\Dx\Dw\M+6\bigl(\bw'\!\bd\bp\sv\bigr)\M
\\
    +6\Dw\bigl(\bw'\!\bd\bp\sv\bigr)\M
    +6\bigl(\bw\bd\bp\sv\bigr)\Dw\M -6\bigl(\bw\bd\bp\x\bigr)\M
    +2\bp\sw\wedge\bigl(\M\bd\bw''\bigr)
\\
    -3\bigl(\bw''\bd\bp\dw\bigr)\bp\dw\otimes\m
    +6\bigl(\bw''\!\bd\bp\sw\bigr)\M +6\bigl(\bw''\!\bd\bp\dw\bigr)\partial_t\M=\,\Zero\,,
\end{multline}
\begin{multline}\label{diss8.5.5}
    -4\bigl(\bw'''\!\bd\bp\sw\bigr)\M
    +3\bigl(\bw''\!\bd\bp\sw\otimes\m\bigr)
    -2\bp\sw\otimes\bigl(\M\bd\bw''\bigr)
    +3\bigl(\bw''\!\bd\bp\dw\bigr)^2\bp\dw\otimes\m
    \\
    \shoveleft{
    -12\Dw\bigl(\bw''\!\bd\bp\sw\bigr)\M +6\Dw\bigl(\bw''\!\bd\bp\dw\bigr)\bp\dw\otimes\m
    }
    +3\bigl(\bw''\!\bd\bp\sw\bigr)\bp\dw\otimes\m
\\
\shoveleft{
    -2\Dw\bp\sw\otimes\bigl(\M\bd\bw''\bigr)
    -2\bigl(\bw''\!\bd\bp\dw\bigr)\bp\sw\otimes\m -4\bigl(\bw''\!\bd\bp\sv\bigr)\M
    +2\bp\sv\otimes\bigl(\M\bd\bw''\bigr)
    }
\\
    -4\Du^3\M +3\Dw^2\bp\dw\otimes\m
    -2\Dw\bp\sw\otimes\m +2\bp\sv\odot\m=\,\Zero
\end{multline}
Careful analysis of these equations implies that $\m$ is a second degree polynomial in $\bw'$,
\begin{equation}\label{matsyuk:m}
\mathsf m_i=\mathsf Q_{ijk}\dw^{j}\dw^{k} + \mathsf N_{ij}\dw^{j} + \mathsf b_i.
\end{equation}
Decompose the matrix $\N$ into the symmetric and skew-symmetric components, $\N=2\PP+\A$. The $\bw'$-independent part of  equation~(\ref{diss8.3.3}) defines the matrix~$\PP$ by the folrmula $\PP=\Du\M$. The $\bw'''$-dependent part of~(\ref{diss8.5.5}) defines the quadratic term in~(\ref{matsyuk:m}),
\begin{equation*}
\Q\bd(\bw'\otimes\bw')=(\bw'\!\bd\bp\sv)\M\bd\bw',
\end{equation*}
and reveals that the matrix $\bp\sw\otimes\M$ is symmetric with respect to the first pair of indices.
Analyzing the $\bw'$-linear terms of~(\ref{diss8.4.4}) and the $\bw''$-linear terms of~(\ref{diss8.5.5}), one concludes that they are equivalent to equation~(\ref{j}).
After some simplifications it turns out that equation~(\ref{diss8.3.3}) is already satisfied, equation~(\ref{diss8.4.4}) coincides with equation~(\ref{ii}), and the $\bw''$-independent and linear in $\bw'$ part of equation~(\ref{diss8.5.5}) coincides with~(\ref{jjj}). The symmetric component of that part of equation~(\ref{diss8.5.5}) that does not contain the variables $\bw'$ and~$\bw''$ coincides with~(\ref{v}), while the skew-symmetric one equals zero according to equation~(\ref{ii}).

The $\bw''$-linear and $\bw'$-independent part of~(\ref{diss8.6})  coincides with~(\ref{i}). Equations \eqref{vi} and~(\ref{iv}) are contained respectively in the part of~(\ref{diss8.6}) that is independent of variables $\bw''$ and~$\bw'$, and in the part dependent on~$\bw''$ and linear in~$\bw'$. The remaining terms in~(\ref{diss8.6}) amount to zero, which finishes the proof.
\end{proof}
\begin{remark}
The differential relation
\begin{equation}\label{diss8.7}
    \lp\sw k(\textup{\ref{v}})_{ij}-(\textup{\ref{iv}})_{kij}+2(\textup{\ref{iv}})_{ikj}-\lp\sw k(\textup{\ref{ii}})_{ij} +2\lp\sv j(\textup{\ref{ii}})_{ik}=0
\end{equation}
between the left-hand sides of equations (\ref{ii}), (\ref{iv}), and~(\ref{v}) holds.
\end{remark}
\begin{remark}
By differentiating equations (\ref{ii}) and~(\ref{j}) with respect to~$\bw$ it is possible to prove that~(\ref{jjj}) is symmetric in the last pair of indexes.
\end{remark}

Actually, the matrices $\M$, $\A$, and the row~$\bb$ are expressed in terms of partial derivatives of the Lagrange function as follows:
\begin{align*}
\sM_{ij} &= \pl\sw i\sw j{}{}\,,\quad \sA_{ij}=\pl\sw i\sv j{}{}-\pl\sw j\sv i{}{},
\\
\Sb_i &=\sw^k\sw^l\pl\sw i\sv k\sv l -\sw^k\pl\sv i\sv k{}{} +2\sw^i\sv^l\pl\sw i\sv k\x l
    +\sw^k\pl\sw i\x k{}{}
\\
& \quad\, +2\sw^k\pl\sw i\sv k{}{}
    +\sv^k\sv^l\pl\sw i\x k\x l -\sv^k\pl\sv i\x k{}{}
    +2\sv^k\pl\sw i\x k t{} +\pl\sw i t{}t{}
    -\pl\sv i t{}{}{} +\pl\x i{}{}{}{}.
\end{align*}
\subsection{The Third  Order Variational ODEs}
\begin{proposition}
\label{prop3order}
A system of the third order ODEs
\begin{equation*}
    \sE_i\left(t,\x^{j},\sv^{j},\sw^{j},\sw'^j\right)=0
\end{equation*}
is a \emph{variational system} if and only if it can be represented as
\begin{equation}\label{matsyuk:Third order shape}
    {\bE}=\A\bd\bw'+\bigl(\bw\bd\bp\sv\bigr)\A\bd\bw+\B\bd\bw+\CC,
\end{equation}
where the skew-symmetric matrix $\A$, the matrix~$\B$ and the row $\CC$ depend on the variables $t$, $\x^{j}$, $\sv^{j}$ and satisfy the following system of PDEs:
\begin{subequations}\label{5}
{\renewcommand{\arraystretch}{1.7}
\begin{align}
\partial_{_{_{_{{\hbox{\sssm v}}}}}}{\!}_{[i}{}{{\sf A}}_{jl]}&=0,
\tag{\theequation i$'$}\label{diss8.i'}
\\
2\,{{\sf B}}_{[ij]}-3\,{\Dx}{\kern.01667em}{{\sf
A}}_{ij}&=0,
\tag{\theequation ii$'$}\label{diss8.ii'}
\\
2\,\partial_{_{_{_{{\hbox{\sssm v}}}}}}{\!}_{[i}{}{{\sf B}}_{j]l}
               -4\,\partial_{_{_{_{{\hbox{\sssm x}}}}}}{\!}_{[i}{}{{\sf
A}}_{j]l}
               +{\partial_{_{_{_{{\hbox{\sssm x}}}}}}{\!}_{l}}{\,}{{\sf A}}_{ij}
             {} +2\,{\Dx}{\kern.01667em}{\partial_{_{_{_{{\hbox{\sssm v}}}}}}{\!}_{l}}{\,}{{\sf
A}}_{ij}&=0,
\tag{\theequation iv$'$}\label{diss8.iv'}
\\
{\partial_{_{_{_{{\hbox{\sssm v}}}}}}{\!}_{(i}}{}{{\sf c}}_{j)}
               -{\Dx}{\kern.01667em}{{\sf B}}_{(ij)}&=0,
               \tag{\theequation v$'$}\label{diss8.v'}
\\
2\,{\partial_{_{_{_{{\hbox{\sssm
v}}}}}}{\!}_{l}}{\,}\partial_{_{_{_{{\hbox{\sssm v}}}}}}{\!}_{[i}{}{{\sf
c}}_{j]}
           -4\,\partial_{_{_{_{{\hbox{\sssm x}}}}}}{\!}_{[i}{}{{\sf B}}_{j]l}
           +{{\Dx}}^{2}{\,}{\partial_{_{_{_{{\hbox{\sssm
v}}}}}}{\!}_{l}}{\,}{{\sf A}}_{ij}
{} +6\,{\Dx}{\kern.0334em}\partial_{_{_{_{{\hbox{\sssm
x}}}}}}{\!}_{[i}{}{{\sf A}}_{jl]}&=0,
\tag{\theequation vi$'$}\label{diss8.vi'}
\\
4\,\partial_{_{_{_{{\hbox{\sssm x}}}}}}{\!}_{[i}{}{{\sf c}}_{j]}
           -2\,{\Dx}{\kern.0334em}\partial_{_{_{_{{\hbox{\sssm
v}}}}}}{\!}_{[i}{}{{\sf c}}_{j]}
           -{{\Dx}}^{3}{\,}{{\sf A}}_{ij}&=0.
           \tag{\theequation vii}\label{diss8.vii}
\end{align}
}
\end{subequations}
\end{proposition}

\begin{proof}
In Proposition~\ref{prop4order} set $\M=\Zero$. In accordance with  condition~(\ref{j}), the matrix~$\A$ depends only on the variables $t$, $\x^i$, and~$\sv^i$. Condition~(\ref{jjj}) specifies the way the row~$\bb$ depends on the variable~$\bw$:
\begin{equation*}
    \bb=\bigl(\bw\!\bd\bp\sv\bigr)\A\bd\bw+\B\bd\bw+\CC.
\end{equation*}
Conditions (\ref{ii}) and~(\ref{iv}) reduce to the desired form.
Condition~(\ref{v}) produces condition~(\ref{diss8.v'}) together with the relation
\begin{equation}\stepcounter{equation}
\lp\sv k{\sf B}_{ij}-\lp\sv{(i}{\sf B}_{j)k}+\Dx\lp\sv{(i}{\sf A}_{j)k}=0\,.
\label{diss8.iii}\tag{\theequation\,iii}
\end{equation}
Condition~{\rm (\ref{vi})} produces equations (\ref{diss8.vi'}) and~(\ref{diss8.vii}). It turns out, however, that~(\ref{diss8.iii}) follows from (\ref{diss8.ii'}) and~(\ref{diss8.iv'}) according to formula~(\ref{diss8.7}) and when taking in consideration the property~(\ref{diss8.i'}) of the matrix~$\A$.
\end{proof}

Actually, the matrix~$\B$ and the row~$\CC$ are expressed in terms of partial derivatives of the Lagrange function as follows:
\renewcommand{\x}{{\mathsf x}}
\renewcommand{\sv}{{\mathsf v}}
\renewcommand{\sv}{{\mathsf v}}
{\setlength{\multlinegap}{0pt}
\begin{align*}
{\sf B}_{ik}
&=
\sw^l\pl\sv i\sv k\sw l -\pl\sv i\sv k{}{} +2\sv^l\pl\sw i\sv k\x l
\\
& \quad\,
   -\sv^l\pl\sw k\sv i\x l +\pl\sw i\x k{}{} -\pl\sw k\x i{}{}
   +2\pl\sw i\sv k t{}
   -\pl\sw k\sv i t{},
\\
{\sf c}_i
&=
\sw^k\sv^l\pl\sv i\sw k\x l +\sw^k\pl\x i\sw k{}{} +\sw\pl\sv i\sw k t{}
\\
& \quad\,
+\sv^k\sv^l\pl\sw i\x k\x l -\sv^k\pl\sv i\x k{}{} +2\sv^i\pl\sw i\x k t{}
    +\pl\sw i t{}t{} -\pl\sv i t{}{}{} +\pl\x i{}{}{}{}.
\end{align*}}

\subsection{Remarks on the First and Second Order Variational ODEs}
To obtain second-order ODEs, set $\A=\Zero$. Then equations \eqref{diss8.ii'}, \eqref{diss8.iv'}--\eqref{diss8.vii} reduce to an equivalent form of the \emph{Helmholtz conditions}~\cite{Engels}.

To obtain first-order ODE's, additionally set $\B=\Zero$. Equations (\ref{diss8.v'}) and~(\ref{diss8.vi'}) imply linear dependence ${\bE}=\bf\varPsi\bd\V+\bm\psi$ with the skew-symmetric matrix~$\bf\varPsi$ depending only on the variables $\V$ and~$t$. Identity~(\ref{diss8.vii}) becomes to the conditions for $\bm{\sE}$ to be self-adjoint~\cite{Hojman}:
\begin{eqnarray*}
  \lp x{[i}\varPsi_{j]k}+\lp x k\varPsi_{ij} &=&\, 0\,, \\
  \bp\x\wedge\bf\varPsi + \partial_t\bf\varPsi&=&\,\Zero\,.
\end{eqnarray*}

\subsection{${\mathbb E^4}$\textup{:} The No-go Theorem}
\index{theorem! no-go}
Consider the (pseudo) Euclidean group of transformations acting on $M=\mathbb E^4=\mathbb R\times \mathbb E^3$.
\begin{proposition}[\cite{thesis}]\label{NoExist}
There are no invariant variational equations of the third order on the (pseudo) Euclidean space of dimension~$4$.
\end{proposition}
\begin{proof}
The symmetry condition~(\ref{7}) with the generator~$X$~(\ref{matsyuk:Lorentz}) splits into separate identities arising from equating the coefficients of the differentials $d\V'\!\!$, $d\V$, $d\X$, $dt$:
\begin{gather}
\label{diss11.1}
    \left[\bP\,\bkey.\,\bp\sv
+(\bP\bcdot \V)\,\V\,\bkey.\,\bp\sv+\W\bcdot(\V\wedge\bp\sv\right]\A
+2\,(\bP\bcdot \V)\,\A+
\A\bkey.\,\V\otimes\bP-\A\,\bkey.\,\W=\bf\varPhi\,\bkey.\,\A
\\\nonumber
2\,(\A\bdot\V')\otimes\bP+(\bP\bdot\V')\,\A=\bf\varPi
\\\nonumber
-\K\otimes\bP=\bf\varXi
\\\nonumber
X\K=\bf\varPhi\bdot\K-\bf\varXi\bdot\V-\bf\varPi\bdot\V'
\end{gather}
Any skew-symmetric matrix of order~$3$ is degenerate. Denoting
\begin{equation*}\label{diss12.9}
\fa\overset{\text{def}}
=\ast\,\A\,,
\end{equation*}
one gets $\A\bdot\fa=0$. Building on this property, contract equation~(\ref{diss11.1}) with the column~$\fa$ and then separate the terms that contain $\bP$ and $\w\overset{\mathrm{def}}=\ast\,\W\,$:
\begin{align}\label{diss13.14}
    \fa\times(\bP\bkey.\,\bpartial_{\sv})\;\fa+(\bP\bcdot\V)\,\fa\times(\V\bkey.\,\bpartial_{\sv})\,\fa
    -(\bP\bcdot\fa)\;\fa\times\V&=0,
\\
    \fa\times[\w\,\V\,\bpartial_{\sv}]\fa-\fa\times(\w\times\fa)&=0.
\label{diss13.15}
\end{align}
In equation~(\ref{diss13.14}), set $\bP=\w\times\V$ and use~(\ref{diss13.15}):
\begin{equation*}
    \fa\times(\w\times\fa) - [\w\,\V\,\fa]\,\fa\times\V=0.
\end{equation*}
Now contract with the column~$\w$:
\begin{equation}\label{matsyuk:diss15'}
    (\fa\times\w)^{\SSS{\boldsymbol2}}+[\fa\,\V\,\w]^{\SSS{2}}=0.
\end{equation}
The quantity $(\fa\times\w)^{\SSS{\boldsymbol2}}=\fa^{\SSS{\boldsymbol2}}\w^{\SSS{\boldsymbol2}}-(\fa\bcdot\w)^{\SSS{2}}$
is positive for the metrics with signature~$\pm3$. In the other case, in view of the fact that $(\fa\times\w)\bcdot\w=0$, it is always possible to choose the column~$\w$ in such a way that the vector $\fa\times\w$ does not point in the imaginary direction, $(\fa\times\w)^{\SSS{\boldsymbol2}}\geq0$. Thus, relation~(\ref{matsyuk:diss15'}) implies that $(\fa\times\w)^{\SSS{\boldsymbol2}}=0$. Since
~$\w$ is arbitrary, $\fa=0$ should hold.
\end{proof}
\subsection{The Invariant Euler--Poisson Equation of a Relativistic Two-dimensional
Motion}
Here we present the proof of Proposition~\ref{matsyuk:E3}.
\begin{proof}
Identity~(\ref{7}), with the generator~$X$ defined by~(\ref{matsyuk:Lorentz}), splits into identities obtained by evaluating the
coefficients of the differentials $d t$, $ d\X$,
$d\V$, and $d\V'$ independently (below we denote the associative matrix multiplication by simple juxtaposition):
\begin{align}
{\big(\bP\,\bmat.\,\bpartial_{\sv}+(\bP\bcdot \V)\,\V\,\bmat.\,\bpartial_{\sv}+\W\bcdot(\V\wedge \bpartial_{\sv})\big)\,\A} &
{}+ 2\,(\bP\bcdot \V)\A+(\A\V)\otimes\bP-\A\W &= {\bf\varPhi}\A;
\label{matsyuk:81}
\\
\label{matsyuk:82}
2\,(\A\V')\otimes\bP+(\bP\bcdot \V')\,\A\,&={\bf\varPi};
\\
\label{matsyuk:83}
-{\K}\otimes\bP &= {\bf\varXi};
\\
\label{matsyuk:84}
{\bf\varPhi} {\K}-{\bf\varXi}\V-{\bf\varPi}\V' &=
X({\K}).
\end{align}
As a non-zero skew-symmetric $2 \times 2$ matrix is always invertible,  the `Lagrange multipliers' ${\bf\varPhi}$, ${\bf\varXi}$, and ${\bf\varPi}$ may be obtained explicitly from equations \eqref{matsyuk:81}--\eqref{matsyuk:83} and then
substituted into~(\ref{matsyuk:84}).  Subsequently, equation
(\ref{matsyuk:84}) splits into the following identities by the powers of the
variable $\V'$ and by the parameters $\W$ and $\bP$ (take notice of the
derivative matrix $\A^{\prime}=(\V'\bmat.\,\bpartial_{\sv})\,\A$; also the
vertical arrow points to the only factor to which the aforegoing
differential operator applies):
\begin{align}
\nonumber
&
\lefteqn{\bigl(\W\bcdot(\V\wedge \bpartial_{\sv})\bigr)\,\A^{\prime}\V'+\bigl(\W\bcdot(\V'\wedge \bpartial_{\sv})\bigr)\,\A\V'
     -(\V'\bmat.\,\bpartial_{\sv})\,\A\W\V'}
\\
\label{matsyuk:95}
&
\quad\,
{{}=\bigl(\W\bcdot(\V\wedge \bpartial_{\sv})\bigr)\,\sa
 \A^{-1}\A^{\prime}\V'-\A\W\A^{-1}\A^{\prime}\V'},
\\
&
\label{matsyuk:96}\stack\bigl(\W\bcdot(\V\wedge \bpartial_{\sv})\bigr)\,\B-\B\W
     =\bigl(\W\bcdot(\V\wedge \bpartial_{\sv})\bigr)\,\sa \A^{-1}\B-\A\W\A^{-1}\B,
\\
&
\label{matsyuk:97}\stack\bigl(\W\bcdot(\V\wedge \bpartial_{\sv})\bigr)\,\CC
     =\bigl(\W\bcdot(\V\wedge \bpartial_{\sv})\bigr)\,\sa \A^{-1}\CC-\A\W\A^{-1}\CC,
\\
\nonumber
&
\lefteqn{\stack\bigl(\bP\,\bmat.\,\bpartial_{\sv}+(\bP\bcdot \V)\,\V\,\bmat.\,\bpartial_{\sv}\bigr)\,\A^{\prime}\V'
     +(\bP\bcdot \V)\,\A^{\prime}\V'+(\bP\bcdot \V')\,(\V\,\bmat.\,\bpartial_{\sv})\,\A\V'+(\bP\bcdot \V')\,\A^{\prime}\V}
\\
&
\quad\,
\label{matsyuk:98}{}=\bigl(\bP\,\bmat.\,\bpartial_{\sv}+(\bP\bcdot \V)\,\V\,\bmat.\,\bpartial_{\sv}\bigr)\,\sa \A^{-1}\A^{\prime}\V'
          +(\bP \A^{-1}\A^{\prime}\V')\,\A\V- 3 \,(\bP\bcdot \V')\,\A\V',
\\
\nonumber
&
\lefteqn{\stack\bigl(\bP\,\bmat.\,\bpartial_{\sv}
     +(\bP\bcdot \V)\,\V\,\bmat.\,\bpartial_{\sv}\bigr)\,\B+(\B\V)\otimes\bP}
\\
&\quad\,
\label{matsyuk:99}
        {}=\bigl(\bP\bcdot\bpartial_{\sv}+(\bP\bcdot \V)\,\V\,\bmat.\,\bpartial_{\sv}\bigr)\,\sa \A^{-1}\B
        +(\A\V)\otimes\bP \A^{-1}\B+(\bP\bcdot \V)\B,
\\
\nonumber
&
\lefteqn{\stack\bigl(\bP\,\bmat.\,\bpartial_{\sv}+(\bP\bcdot \V)\,\V\,\bmat.\,\bpartial_{\sv}\bigr)\,\CC}
\\
&\quad\,
\label{matsyuk:90}{}=\bigl(\bP\,\bmat.\,\bpartial_{\sv}+(\bP\bcdot \V)\,\V\,\bmat.\,\bpartial_{\sv}\bigr)\,\sa \A^{-1}\CC
     + 3 \,(\bP\bcdot \V)\,\CC+(\bP \A^{-1}\CC)\,\A\V.
\end{align}
Let
\[
\A=
\begin{pmatrix}
\hphantom{-}
0&\quad &a\\
-a&\quad &0
\end{pmatrix},
\]
so that $\A^{\prime}=\frac{1}{a}\A\,(\V'\bmat.\,\bpartial_{\sv})\,a$. In~(\ref{matsyuk:95}) the terms  $-(\V'\bmat.\,\bpartial_{\sv})\,\A\W\V'$ and $-\A\W\A^{-1}\A^{\prime}\V'$ cancel out (here we notice the first occurrence of a point where the dimension~$2$ plays an important role).
Further splitting by the components of vector~$\V'$ of either the first or the second component of the two-component equation~(\ref{matsyuk:95}) with respect to either the terms $\sv_1\sv_2$ and $(\sv_2)^2$ or $\sv_1\sv_2$ and $(\sv_1)^2$ produces at most the following two independent equations with second order partial derivatives of~$a$ with respect to $\sv_1$ and $\sv_2$:
\begin{subequations}
\renewcommand{\theequation}{\theparentequation.\Roman{equation}}\label{matsyuk:dadpp_from_Omega}
\begin{align}\label{matsyuk:dadpp1_from_Omega}
    \bigl(\sv_1\partial_{\sv_{_2}}\nsp-\sv_2\partial_{\sv_{_1}}\bigr)\,\partial_{\sv_{_1}}a+\partial_{\sv_{_2}}a
    &
    =\frac{\partial_{\sv_{_1}}a }{a}\bigl(\sv_{_1}\partial_{\sv_{_2}}-\sv_2\partial_{\sv_{_1}}\bigr)\,a,
    \\\label{matsyuk:dadpp2_from_Omega}
    \bigl(\sv_1\partial_{\sv_{_2}}\nsp-\sv_2\partial_{\sv_{_1}}\bigr)\,\partial_{\sv_{_2}}a-\partial_{\sv_{_1}}a
    &
    =\frac{\partial_{\sv_{_2}}a}{a}\bigl(\sv_1\partial_{\sv_{_2}}-\sv_2\partial_{\sv_{_1}}\bigr)\,a.
\end{align}
\end{subequations}

In a similar way, the terms $\,(\bP\bcdot \V')\,\A^{\prime}\V\,$ and $\,(\bP \A^{-1}\A^{\prime}\V')\,\A\V\,$ cancel out in equation~(\ref{matsyuk:98}). Again, one may focus on only one component of equation~(\ref{matsyuk:98}), and, after splitting with respect to~$\pi_i\sv_j$, get the following system of independent equations:
\newcommand{\pdp}[1][]{(\V\,\bmat.\,\bpartial_{\sv})\,#1}
\newcommand{\pDp}{\V\,\bmat.\,\bpartial_{\sv}}
\newcommand{\fp}[2]{\partial_{\sv_{_#1}}\negmedspace\frac{\partial_{\sv_{_#2}}a}{a}\,}
\newcommand{\pp}[1]{\partial_{\sv_{_#1}}}
\begin{subequations}\renewcommand{\theequation}
{\theparentequation.\Roman{equation}}\label{matsyuk:dadpp_from_pi}
\begin{align}
\label{matsyuk:dadpp11_from_pi}
 \bigl(\pp1\nsp-\sv_1\pdp\bigr)\,\pp1a -\sv_1\pp1a -\pdp a
    &=\frac{\pp1a}{a}\, \bigl(\pp 1\nsp-\sv_1\pdp\bigr)\,a+3a,
\\
 \label{matsyuk:dadpp22_from_pi}
 \bigl(\partial_{\sv_{_2}}\nsp-\sv_2(\V\,\bmat.\,\bpartial_{\sv})\bigr)\,\partial_{\sv_{_2}}a -\sv_{_2}\partial_{\sv_{_2}}a-(\V\,\bmat.\,\bpartial_{\sv})\,a
     &
   =\frac{\partial_{\sv_{_2}}a}{a}\,\bigl(\partial_{\sv_{_2}}\nsp-\sv_2(\V\,\bmat.\,\bpartial_{\sv})\bigr)\,a+3a,
 \\
 \label{matsyuk:dadpp21_from_pi}
 \bigl(\pp2\nsp-\sv_2\pdp\bigr)\,\pp1a -\sv_2\pp1a
    &=\frac{\pp1a}{a}\, \bigl(\pp2\nsp-\sv_2\pdp\bigr)\,a ,
    \\\label{matsyuk:dadpp12_from_pi}
 \bigl(\pp1\nsp-\sv_1\pdp\bigr)\,\pp2a -\sv_1\pp2a
    &=\frac{\pp2a}{a}\, \bigl(\pp 1\nsp-\sv_1\pdp\bigr)\,a.
\end{align}
\end{subequations}

Next, proceed with solving the system of equations \eqref{matsyuk:dadpp_from_Omega} and \eqref{matsyuk:dadpp_from_pi}. Multiplying~(\ref{matsyuk:dadpp1_from_Omega}) on the left by~$\sv_1$ and adding to~(\ref{matsyuk:dadpp2_from_Omega}) multiplied on the left by~$\sv_2$, we obtain (after some transformations)
\begin{equation}\label{matsyuk:0}
    \bigl(\sv_1\pp2\nsp-\sv_2\pp1\negthickspace\bigr)\pdp a= \frac{\pdp a}{a}\cdot\bigl(\sv_1\pp2\nsp-\sv_2\pp1\negthinspace\bigr)\,a.
\end{equation}
Introduce the notation
\begin{equation}\label{matsyuk:f}
    f = \frac{\pdp a}{a}.
\end{equation}
Equation~(\ref{matsyuk:0}) may be rewritten as
\begin{equation}\label{matsyuk:p^dpf=0}
\bigl(\sv_1\pp2\nsp-\sv_2\pp1\negthinspace\bigr)\,f=0.
\end{equation}
The above condition means that the function~$f$ depends on $\sv_1$ and $\sv_2$ through the variable
\begin{equation}\label{matsyuk:y}
y=\sv_1{}^2+\sv_2{}^2.
\end{equation}

Add (sidewise) equations (\ref{matsyuk:dadpp11_from_pi}) and ~(\ref{matsyuk:dadpp22_from_pi}). After some trasformations, the result reads
\begin{multline*}
    \pp1\pp2a+\pp2\pp1a-\pdp^2a-2\pdp\,a=
    \\
    \frac{1}{a}\bigl[(\pp1a)^2+(\pp2a)^2-(\pDp\; a)^2\bigr]+6a.
\end{multline*}
Using~(\ref{matsyuk:f}), this becomes
\begin{equation}\label{matsyuk:(1)}
   \fp12+\fp21=2f+2yf_y'+6.
\end{equation}

Add equations (\ref{matsyuk:dadpp12_from_pi}) and ~(\ref{matsyuk:dadpp21_from_pi}). Again, after some transformations, the result reads
\begin{multline*}
    2\,\pp1\pp2a-\bigl(\sv_1\pp2\!+\sv_2\pp1\!\bigr)\pdp a=
    \\
    \frac{1}{a}\,\bigl[2\,(\pp1a)\cdot(\pp2a)
    -\bigl(\sv_1\pp2a+\sv_2\pp1a\bigr)\cdot(\pDp\:a)\bigr],
\end{multline*}
and, again, using~(\ref{matsyuk:f}), it becomes
\begin{equation}\label{matsyuk:(2)}
    \fp12\equiv\fp21=2\,\sv_1\sv_2f_y'.
\end{equation}

In a similar manner, subtract~(\ref{matsyuk:dadpp22_from_pi}) from~(\ref{matsyuk:dadpp11_from_pi}):
\begin{multline}
\label{matsyuk:80}
\pp1\pp1a-\pp2\pp2a-\bigl(\sv_1\pp1\nsp-\sv_2\pp2\!\bigr)\pdp a
=
\\
\frac{1}{a}\,\bigl[(\pp1a)^2-(\pp2a)^2
-\bigl(\sv_1\pp1a-\sv_2\pp2a\bigr)\cdot(\pDp\:a)\bigr].
\end{multline}
Using~(\ref{matsyuk:f}) we calculate the derivatives,
\begin{equation*}
   \bigl(\sv_1\pp1\nsp-\sv_2\pp2\!\bigr)\pdp a = f\cdot\bigl(\sv_1\pp1\nsp-\sv_2\pp2\!\bigr)\,a+2\,\,a\bigl(\sv_1{}^2-\sv_2{}^2\bigr)\,f_y',
\end{equation*}
and substitute in \eqref{matsyuk:80} to obtain
\begin{equation}\label{matsyuk:(3)}
    \fp11-\fp22=2\,(\sv_1{}^2-\sv_2{}^2)f_y'.
\end{equation}

Subtracting equation \eqref{matsyuk:dadpp21_from_pi} from \eqref{matsyuk:dadpp12_from_pi}, we obtain
\begin{equation}\label{matsyuk:IV}
\pdp\bigl(\sv_1\pp2\nsp-\sv_2\pp1\!\bigr)\,a=\frac{\pdp a}{a}\cdot\bigl(\sv_1\pp2\nsp-\sv_2\pp1\!\bigr)\,a.
\end{equation}
Let
\begin{equation}\label{matsyuk:zeta}
\zeta=\bigl(\sv_1\pp2\nsp-\sv_2\pp1\!\bigr)\,a\,.
\end{equation}
Equation~(\ref{matsyuk:IV}) means that
\begin{equation}\label{matsyuk:pdp zeta/a=0}
     \pdp{\frac{\zeta}a}=0.
\end{equation}

Now calculate the left hand side of the above equation. Using~(\ref{matsyuk:zeta}), (\ref{matsyuk:(2)}), and (\ref{matsyuk:(3)}) we obtain:
\begin{equation*}
\begin{aligned}
      &\pdp{\frac{\zeta}a}  = \bigl(\sv_1\pp1\negmedspace+\sv_2\pp2\!\bigr)\left(\frac{\sv_1\pp2 a}{a} -\frac{\sv_2\pp1 a}{a}\right)\\
   &\quad= \sv_1{}^2\fp12+\sv_1\frac{\pp2a}{a}+\sv_1\sv_2\,\fp22-\sv_1\sv_2\,\fp11 -\sv_2{}^2\fp21-\sv_2\frac{\pp1a}{a}\\
        &\quad=\sv_1\sv_2\left(\pp2\negmedspace\frac{\pp2a}{a}-\pp1\negmedspace\frac{\pp1a}{a}\right)+2\,\sv_1{}\bigl(\sv_1{}^2-\sv_2{}^2\bigr)\,f_y'+\frac{\zeta}{a}
\\
&\quad=\frac{\zeta}{a}.
\end{aligned}
\end{equation*}
Thus,~(\ref{matsyuk:pdp zeta/a=0}) implies
\begin{equation*}
    \zeta=0\,.
\end{equation*}
This means that the function~$a$ depends on~$\sv_1$ and $\sv_2$ the same way the function~$f$ does, that is, only through the combination~(\ref{matsyuk:y}).

The last step is to utilize equations~(\ref{matsyuk:(1)}) and~(\ref{matsyuk:(2)}) to identify the function~$a$. From~(\ref{matsyuk:f}), one calculates:
\begin{align}\label{matsyuk:a'/a}
\frac{f}{y}&=2\,\frac{a_y'}{a},
\\[2\jot]
\nonumber
\frac{\pp1a}{a}&=\frac{\sv_1f}{y},
\quad
\frac{\pp2a}{a}=\frac{\sv_2f}{y},
\\[2\jot]
\nonumber
\fp11&=\frac{f}{y}+2\,\sv_1{}^2\left(\frac{f_y'}{y}-\frac{f}{y^2}\right),
\quad
\fp22=\frac{f}{y}+2\,\sv_2{}^2\left(\frac{f_y'}{y}-\frac{f}{y^2}\right),
\\[2\jot]
\nonumber
\fp12&=2\,\sv_1{}\sv_2\left(\frac{f_y'}{y}-\frac{f}{y^2}\right).
\end{align}
Next, equations~(\ref{matsyuk:(1)}) and~(\ref{matsyuk:(2)}) become
\begin{equation}\label{matsyuk:(1*)}
    (1-y)f_y'-f=3,
\end{equation}
and
\begin{equation*}
    \frac{f_y'}{y}-f_y'-\frac{f}{y^2}=0.
\end{equation*}
The latter implies
\begin{equation*}
    f=\text{const}\cdot\frac{y}{1-y},
\end{equation*}
and from~(\ref{matsyuk:(1*)}) one concludes that the constant equals~$3$.
Finally, solving the differential equation~(\ref{matsyuk:a'/a}), we get (recall the signature of the space)
\begin{equation}\label{matsyuk:aFinal}
a = \sA_{12}=\frac{\text{const}}{\left(1-\sv_1{}^2-\sv_2{}^2\right)^{3/2}}\equiv\frac{\text{const}}{\left(1+\sv_1\sv^1+\sv_2\sv^2\right)^{3/2}},
\end{equation}
which defines the first term in~(\ref{10}) (or~(\ref{matsyuk:Third order shape})). This completes the first part of the proof.
It is worthwhile to recall that the system of equations (\ref{matsyuk:95})--(\ref{matsyuk:90}) should be solved simultaneously with system~(\ref{5}). But equation~(\ref{diss8.i'}) is satisfied trivially as long as~$\sA_{12}$ is taken from~(\ref{matsyuk:aFinal}).

Next, we identify the matrix~$\B$. Under the assumption of $\B$ being a symmetric matrix (see
(\ref{diss8.ii'})),
the solution of equations \eqref{matsyuk:96} and \eqref{matsyuk:99} is
\[
{\sB}_{ij}={\rm const}\cdot(1+\V\bcdot \V))^{-3/2}\bigl({v}_{i}{v}_{j}
-(1+\V\bcdot \V)\,{g}_{ij}\bigr).
\]
This automatically satisfies equation (\ref{diss8.iv'}) too.
As for the subsystem \eqref{matsyuk:97} and \eqref{matsyuk:90},
only the trivial solution
$\CC={\Zero}$ exists.
\end{proof}

\balance
\printindex
\end{document}